\documentclass[11pt]{amsart}
\usepackage{amsfonts}
\usepackage{amsxtra}
\usepackage{amstext}
\usepackage{amsbsy}
\usepackage{amsthm}
\usepackage{latexsym}
\usepackage{amscd}
\usepackage{eucal}
\usepackage{mathrsfs}
\usepackage{amsmath,amssymb}
\usepackage[activeacute,english]{babel}
\usepackage{graphicx}
\usepackage[latin1]{inputenc}

\theoremstyle{plain}
\newtheorem{theorem}{Theorem}[section]
\newtheorem{proposition}[theorem]{Proposition}
\newtheorem{lemma}[theorem]{Lemma}

\theoremstyle{definition}
\newtheorem{definition}[theorem]{Definition}
\newtheorem{example}[theorem]{Example}
\newtheorem{observation}[theorem]{Observation}

\begin{document}
\title[]{INVERSION OF DOUBLE FOURIER INTEGRAL OF NON-LEBESGUE
INTEGRABLE BOUNDED VARIATION FUNCTIONS}

\author[E. Torres-Teutle, M. G. Morales-Mac\'ias and F. J. Mendoza-Torres]{
Edgar Torres-Teutle$^{1}$, Francisco J. Mendoza-Torres$^{1}$ \\ and Mar\'ia G. Morales-Mac\'ias$^{1,*}$}
\maketitle
\begin{center}
\address{\textbf{$^{1}$} Facultad de Ciencias F\'isico Matem\'aticas, Benem\'erita Universidad Aut\'onoma de Puebla, Av. San Claudio and 18 Sur, Puebla, 72570, Mexico.}   
\end{center}

\address{$^{*}$Corresponding author: \texttt{maciam@math.muni.cz}}  

\maketitle

\begin{abstract}
This work proves pointwise convergence of the truncated Fourier double integral of non-Lebesgue integrable bounded variation functions. This leads to the Dirichlet-Jordan theorem proof for non-Lebesgue integrable functions, which has not been sufficiently studied. Note that recent contributions regarding this subject consider Lebesgue integrable functions, [F. M\'oricz, 2015], [B. Ghodadra-V. F$\ddot{\text{u}}$lop, 2016].\\


    
\end{abstract}

\maketitle

\section{Introduction}

One of the most relevant and significant subjects in the Fourier analysis
theory is the inversion problem. This means, given the Fourier transform $%
\hat{f}$ of a function $f$ on $\mathbb{R}^{n}$, provides conditions such that
the function 
\begin{equation*}
C\int_{\mathbb{R}^{n}}\hat{f}(\omega )e^{i<\omega ,\overline{x}>}d\omega ,\;\;\overline{x}\in 
\mathbb{R}^{n},
\end{equation*}%
approximates to $f(\overline{x})$,
where $C$ is a normalization constant and $<,>$ is the Euclidean inner
product.

The Dirichlet-Jordan Theorem solves the pointwise inversion problem.
For $n=1$, this states that if $\ f\in L^{1}(\mathbb{R})\cap BV(\mathbb{R})$ then, for
each $x\in \mathbb{R}$,%
\begin{equation}\label{CDJ}
\lim_{M\rightarrow \infty }\frac{1}{2\pi }\int_{-M}^{M}\hat{f}(\omega )e^{ix\omega }d\omega=\frac{f(x+)+f(x-)}{2}.
\end{equation}%
The integral at left side of (\ref{CDJ}) is called the truncated Fourier integral, also known as the Dirichlet integral of $f$.
 In \cite[Corollary 3]{Moricz1}, F. M\'oricz proved that if $f\in L^{1}(\mathbb{R})\cap BV(\mathbb{R})$, then the convergence (\ref{CDJ})
is locally uniform at every point of continuity of $f$.

For the case $n=2$, considering the classical Lebesgue integral
theory, in \cite{Moricz2}, F. M\'oricz proved locally uniform convergence of the truncated double Fourier integral 
\begin{equation*}
\frac{1}{%
4\pi ^{2}}\int_{|\xi |\leq u}\int_{|\eta|\leq v}
\hat{f}(\xi ,\eta )e^{i(\xi x+\eta y)}d(\xi ,\eta ),
\end{equation*}%
to $f\left( x,y\right) ,$ as $u,v\rightarrow \infty $, under the
conditions: $f\in L^{1}(\mathbb{R}^{2})\cap BV_{H}(\mathbb{R}^{2})$,  
\begin{equation}
\hat{f}\in L^{1}((\mathbb{R})\times \lbrack -\delta ,\delta ])\cup ([-\delta
,\delta ]\times \mathbb{R}),\;\;\delta >0,  \tag{C1}
\label{condition integrability FT}
\end{equation}%
and $(x,y)\in\mathbb{R}^{2}$ a point of continuity of $f$. The set of bounded variation functions in the sense of
Hardy over $\mathbb{R}^{2}$ is denoted as $BV_{H}(\mathbb{R}^{2})$.  In \cite{Ghodadra}, 
it is proved that the integrability condition \eqref{condition integrability
FT} about $\hat{f}$ can be omitted to get locally uniform convergence  of the
truncated double Fourier integral.

We have that
\begin{equation}
L^{1}(\mathbb{R}^{2})\cap BV_{H}(\mathbb{R}^{2})\subsetneq BV_{||0||}(%
\mathbb{R}^{2})\nsubseteq L^{1}(\mathbb{R}^{2}),  \label{no contension}
\end{equation}%
where $BV_{||0||}(\mathbb{R}^{2})$ denotes the set of bounded variation
functions in the sense of Vitali that vanish when the norm of their
arguments tends to infinity. Thus, previous results and expression \eqref{no
contension} motivate us to consider the set $BV_{||0||}(\mathbb{R}^{2})$ to
study the inversion problem. 

The relations in  \eqref{no contension} presupposes the use of integrals other than the Lebesgue one. Here, we consider locally
Kurzweil-Henstock integrable functions over $\mathbb{R}^{2}$. Thus, we
will show that if $f\in BV_{||0||}(\mathbb{R}^{2})$ and $(\xi ,\eta )\in 
\mathbb{R}^{2},$ where $\xi \neq 0$ and $\eta \neq 0$, then the map 
\begin{equation*}
(\xi ,\eta )\longrightarrow \lim\limits
_{\substack{ a,c\rightarrow -\infty  \\ b,d\rightarrow \infty }}%
\iint_{[a,b]\times \lbrack c,d]}f(t_{1},t_{2})e^{-i(\xi t_{1}+\eta
t_{2})}d(t_{1},t_{2})  \label{limite}
\end{equation*}%
is well defined. We call this limit the KP-Fourier transform of $f$ at $%
(\xi ,\eta )$ and denote by $\mathcal{F}(f)(\xi ,\eta )$. Of course, the
KP-Fourier transform is defined in a more general sense than the classical Fourier transform, see Definition \ref%
{Definition P-Fourier}.

One important result in this article is the Dirichlet-Jordan theorem for the KP-Fourier transform. That is, if $f\in BV_{||0||}(\mathbb{R}^{2})$, then, for $x\neq 0$ and $y\neq 0$, 
\begin{equation}\label{tdFi}
\frac{1}{%
4\pi ^{2}}\int_{\alpha _{1}\leq |\xi |\leq \beta _{1}}\int_{\alpha _{2}\leq
|\eta |\leq \beta _{2}}\mathcal{F}(f)(\xi ,\eta )e^{i(\xi x+\eta
y)}d(\xi ,\eta )
\end{equation}%
converges pointwise  to%
\begin{equation*}
\frac{f(x+,y+)+f(x+,y-)+f(x-,y+)+f(x-,y-)}{4},
\end{equation*}%
as $\alpha _{1},\alpha _{2}\rightarrow 0 \text{ and } \beta _{1},\beta _{2}\rightarrow
\infty $. Apparently, in mathematical literature there is no a similar space on which the proof of this theorem has been analyzed.

This article is organized as follows. In Section 2, we present the improper
Riemann-Stieltjes integral definition over $\mathbb{R}^{2}$ and some of its
properties, the concepts of bounded variation in the sense of
Vitali and Hardy. In Section 3,
we recall the Kurzweil-Henstock integral over rectangles. Also, we introduce
the definition of the KP-Fourier transform which was defined in \cite{Mendoza3}, and we provide an alternative proof of its continuity property which was demonstrated in \cite{Mendoza3} and some auxiliary results. In section 4, we present our main
contributions; we prove a version of the Dirichlet-Jordan Theorem of non-Lebesgue integrable bounded variation functions, see Theorem \ref{TDJI}. Moreover, we extend Theorem 1 from \cite{Moricz1} and Theorem 2.1 in \cite{Ghodadra}. This leads us to consider the validity of the locally uniform convergence for the truncated double Fourier integral of functions in $BV_{||0||}(\mathbb{R}^2)$.

\section{Preliminary topics}

Let us recall that a partition of the bounded interval $[a,b]$ is a finite
collection of subintervals $\{[x_{i-1},x_i]:i=1,...,n\},$ where $a=x_0<...<x_n=b$. 
Now, we consider $R=[a,b]\times [c,d]$ a bounded rectangle of $\mathbb{R}^2$. A partition of $R$ is a finite collection of the form $$\{R_{i,j}\}=\{[x_{i-1},x_i]\times[y_{j-1},y_j]\text{ }|\text{ }i=1,..,n\text{ and }j=1,...,m\},$$
where $a=x_0<...<x_n=b\text{ and }c=y_0<...<y_m=d$,
and the set of such partitions is defined by $\mathcal{P}(R)$. Furthermore, we define the norm of $%
P=\{R_{i,j}\}\in \mathcal{P}(R)$ as 
\begin{equation}  \label{ENorm}
||P||=\max\{D_{i,j}\},
\end{equation}
where $D_{i,j}$ is the diagonal length of $R_{i,j}$.

\begin{definition}
Let $f,g:R\rightarrow\mathbb{R}$ be functions. It is said that $g$ is 
\textbf{Riemann-Stieltjes integrable} with respect to $f$ over $R$, if there
exists $A\in \mathbb{R}$ such that for $\varepsilon>0$ there exists $%
\delta_\varepsilon>0$ such that if $P=\{R_{i, j}\}$ with $%
||P||<\delta_\varepsilon$ and $(\eta_{i,j},\tau_{i,j})\in R_{i,j}$, then 
\begin{equation*}
\left|\sum_{i,j} g(\eta_{i,j},\tau_{i,j})\Delta
f(x_i,y_j)-A\right|<\varepsilon,
\end{equation*}
where 
\begin{equation*}
\Delta
f(x_i,y_j):=f(x_i,y_j)-f(x_{i-1},y_j)-f(x_i,y_{j-1})+f(x_{i-1},y_{j-1}).
\end{equation*}
Further, $A$ is the value of the Riemann-Stieltjes integral and is denoted
as 
\begin{equation*}
A=\iint_Rgdf=\iint_R g(t_1,t_1)\,df(t_1,t_2).
\end{equation*}
\end{definition}

\begin{observation}\noindent
\begin{itemize}
\item[\textit{i)}] The norm in $(\ref{ENorm})$ is equivalent to the norm 
 used in \cite{Clarkson} and  \cite{Ghodadra}, hence the integrals defined by each norm are
equal. 

\item[\textit{ii)}] The Riemann-Stieltjes integral of a complex function $g$ with respect
to a real valued function $f$ is the sum of the integrals of the real and
imaginary part of $g$ with respect to $f$.
\end{itemize}
\end{observation}

From \cite[Theorem 5.4.3]{Giselle} for functions of one variable, we obtain
the following result for functions of two variables.

\begin{lemma}
\label{IIIR} Let f be a Riemann integrable function over $R$
and suppose that $g$ is bounded and Riemann-Stieltjes integrable with respect to $h$ over $R$,
where 
\begin{equation*}
h(t_1,t_2)=\iint_{[a,t_1]\times[c,t_2]} f(s_1,s_2)d(s_1,s_2),
\end{equation*}
for each $(t_1,t_2)\in R$.  Then 
\begin{equation*}
\iint_{R}g(t_1,t_2)dh(t_1,t_2) =\iint_{R}g(t_1,t_2)f(t_1,t_2)d(t_1,t_2).
\end{equation*}
\end{lemma}

In the following definition, we will consider rectangles $Q$ where their
sides $I_1$ and $I_2$ can be of the form $(-\infty,\infty)$, $[a,\infty)$, $%
(-\infty,a]$ or bounded intervals. In addition, we will use the concept of
bounded variation on intervals, \cite{Moricz1}.

\begin{definition}\label{DFVA}\noindent
\begin{itemize}
\item[\textit{i)}] A function $f:Q\rightarrow\mathbb{R}$ is said to be of \textbf{%
bounded variation in the Vitali sense} over $Q$ and is denoted as $f \in
BV_V(Q)$, if  
\begin{equation*}
Var(f,Q):=\sup_{ R\subset Q}\text{ }\sup_{\{R_{i,j}\}\in \mathcal{P}(R)}
\left \{\sum_{i,j} \left |\Delta f(x_i,y_j) \right | \right \}<\infty,
\end{equation*}
where the rectangles $R$ are compacts contained in $Q$.

\item[\textit{ii)}] A function $f:Q\rightarrow\mathbb{R}$ is said to be of \textbf{%
bounded variation in the Hardy sense} over $Q$ and is denoted as $f\in
BV_H(Q)$, if $f\in BV_V(Q)$ and, for each $x,y$, $f(\cdot,y)$ y $f(x,\cdot)$
are of bounded variation over $I_1$ and $I_2$, respectively.
\end{itemize}
\end{definition}

It is evident that $BV_{H}(\mathbb{R}^2)$ is properly contained in $BV_{V}(%
\mathbb{R}^2)$ since $f(x,y)=x+y\in BV_{V}(\mathbb{R}^2)\setminus BV_{H}(%
\mathbb{R}^2)$, \cite{Moricz2}.

\begin{observation}
\label{Tlim} We can observe that if $f\in BV_V(\mathbb{R}^2)$, then 
\begin{equation*}
\begin{split}
\lim \limits_{\substack{ m \to\infty}}(Var(f,[m,\infty)\times\mathbb{R}%
)+Var(f,(-\infty,-m]\times\mathbb{R})& \\
+Var(f,\mathbb{R}\times[m,\infty))+Var(f,\mathbb{R}\times(-\infty,-m]))&=0.
\end{split}%
\end{equation*}
\end{observation}

We will show some properties that relate the previous concepts, but first we
recall the following spaces of continuous functions defined on $\mathbb{R}^2$%
.

\begin{itemize}
\item[\textit{i)}] $C_b(\mathbb{R}^2)$ is the space of real valued functions which are
bounded. 

\item[\textit{ii)}] $C_0(\mathbb{R}^2)$ is the space of functions which vanish at infinity. 

\item[\textit{iii)}] $C_c(\mathbb{R}^2)$ is the space of functions whose support is compact.
\end{itemize}

It is well known that $C_c(\mathbb{R}^2)\subsetneq C_0(\mathbb{R}%
^2)\subsetneq C_b(\mathbb{R}^2)$ and $C_c(\mathbb{R}^2)$ is dense in $C_0(%
\mathbb{R}^2)$ with respect to the supremum norm, \cite[Theorem 3.17]{Rudin}.

\begin{lemma}[{\textbf{\protect\cite[Lema 3.9]{Ghodadra}}}]
\label{TG2} Let $f\in BV_V(\mathbb{R}^2)$ and $(x,y)\in \mathbb{R}^2$ be a
continuity point of $f$. Then for each $\varepsilon>0$, there exists $%
\delta>0$ such that 
\begin{equation*}
Var(f,[x-\delta,x+\delta]\times[y-\delta,y+\delta])<\varepsilon.
\end{equation*}
\end{lemma}

Now we define the improper Riemann-Stieltjes integral.

\begin{definition}
Let $f,g:\mathbb{R}^2\rightarrow\mathbb{R}$ be functions. The multiple limit 
\begin{equation*}  \label{LPrin}
\lim_{\begin{subarray}{l}a,c\rightarrow -\infty\\
b,d\rightarrow\infty\end{subarray}}\iint_{[a,b]\times[c,d]%
}g(t_1,t_2)df(t_1,t_2)
\end{equation*}
is called the \textbf{improper Riemann-Stieltjes integral of $g$ with
respect to $f$}, when it exists in the Pringsheim sense \cite{Pringsheim},
and is denoted by%
\begin{equation*}
\iint_{\mathbb{R}^2}g(t_1,t_2)df(t_1,t_2).
\end{equation*}
\end{definition}

From \cite{Moricz2}, if $g\in C_b(\mathbb{R}^2)$ and $f\in BV_V(\mathbb{R}^2)
$, then the improper Riemann-Stieltjes integral of $g$ with respect to $f$
exists.

\begin{lemma}[{\textbf{\protect\cite[Lema 3.6]{Ghodadra}}}]
\label{TG1} If $g\in C_b(\mathbb{R}^2)$ and $f\in BV_V(\mathbb{R}^2)$, then 
\begin{flalign*}
&\begin{aligned}
\text{i)}&\text{ The function } V(f;t_1,t_2)=Var(f,(-\infty,t_1]\times(-\infty,t_2)) \text{ defined for each }\\
& \text{  }(t_1,t_2)\in\mathbb{R}^2,
\text{ belongs to } BV_V(\mathbb{R}^2)\\
\text{ii)}&\text{ } \left|\iint_{\mathbb{R}^2}g(t_1,t_2)df(t_1,t_2)\right|\leq\iint_{\mathbb{R}^2}|g(t_1,t_2)|dV(f;t_1,t_2).
\end{aligned}&&
\end{flalign*}
\end{lemma}

\begin{theorem}
\label{TY1} Suppose that $g\in C_0(\mathbb{R}^2)$ and $f\in BV_V(\mathbb{R}%
^2)$. Then, there exists a $\sigma$-algebra $\mathbb{M}$ of \text{ }$\mathbb{R}^2$
containing the Borelians and there exists a unique positive finite measure $%
\mu$ on $\mathbb{M}$ such that 
\begin{flalign*}
&\begin{aligned}
\text{i)}&\textbf{ } \mu((a,b)\times(c,d))\leq Var(f,[a,b]\times[c,d]), \text{ for } (a,b)\times(c,d)\subset\mathbb{R}^2.\\
\text{ii)}&\text{ } \iint_{\mathbb{R}^2}g(t_1,t_2)dV(f;t_1,t_2)=\iint_{\mathbb{R}^2}g(t_1,t_2)d\mu(t_1,t_2).
\end{aligned}&&
\end{flalign*}
\end{theorem}

\begin{proof}
From Lemma \ref{TG1}, over $C_b(\mathbb{R}^2)$ we define the bounded
positive linear functional 
\begin{equation*}
\Lambda(g)=\iint_{\mathbb{R}^2} g(t_1,t_2)dV(f;t_1,t_2).
\end{equation*}

By \cite[Theorem 2.14]{Rudin} (Riesz Representation Theorem), there exists a 
$\sigma$-algebra $\mathbb{M}$ of $\mathbb{R}^2$ containing the Borelians and
there exists a unique positive measure $\mu$ on $\mathbb{M}$ such that 
\begin{equation}
\begin{split}
\Lambda(g)=\iint_{\mathbb{R}^2}g(t_1,t_2)d\mu(t_1,t_2), \text{ for all }
g\in C_c(\mathbb{R}^2).  \label{TRU}
\end{split}%
\end{equation}

Let $R=[a,b]\times[c,d]\subset \mathbb{R}^2$. From \cite[Theorem 2.14]{Rudin}%
, it follows that $\mu(int(R))\leq Var(f,R)$. Thus, $\mu$ is a finite
measure. This proves $i)$.\newline

Now, let $g\in C_0(\mathbb{R}^2)$ be. Then, there exists a sequence of
functions $(g_n)$ that belongs to $C_c(\mathbb{R}^2)$ which converges
uniformly to $g$. Moreover, 
\begin{equation}
\begin{split}
\lim_{n\rightarrow\infty}\iint_{\mathbb{R}^2}g_n(t_1,t_2)dV(f;t_1,t_2)=%
\iint_{\mathbb{R}^2}g(t_1,t_2)dV(f;t_1,t_2).  \label{ETR1}
\end{split}%
\end{equation}
and 
\begin{equation}
\begin{split}
\lim_{n\rightarrow\infty}\iint_{\mathbb{R}^2}g_n(t_1,t_2)d\mu(t_1,t_2)=%
\iint_{\mathbb{R}^2}g(t_1,t_2)d\mu(t_1,t_2).  \label{ETR2}
\end{split}%
\end{equation}

By (\ref{TRU}), (\ref{ETR1}) and (\ref{ETR2}) we conclude that, for $g\in
C_0(\mathbb{R}^2)$, 
\begin{equation*}
\begin{split}
\iint_{\mathbb{R}^2}g(t_1,t_2)dV(f;t_1,t_2)=\iint_{\mathbb{R}%
^2}g(t_1,t_2)d\mu(t_1,t_2).
\end{split}%
\end{equation*}
\end{proof}

\subsection{The space $BV_{||0||}(\mathbb{R}^2)$}

The space of functions $f$ in $BV_V(\mathbb{R}^2)$ which satisfy that $%
\lim\limits_{\substack{ ||(x,y)|| \to \infty}}f(x,y)=0$ is denoted by $%
BV_{||0||}(\mathbb{R}^2)$. 

The following lemma can be proved from the definition of bounded variation, see Definition \ref{DFVA} and \cite{Moricz1}.

\begin{lemma}
\label{TA1} Let $f\in BV_{||0||}(\mathbb{R}^2)$ be. Then, for each $(x,y)\in%
\mathbb{R}^2$, 
\begin{flalign*}
&\begin{aligned}
\text{i)}&\textbf{ } |f(x,y)|\leq Var(f(x,\cdot),\mathbb{R})\leq Var(f,[x,\infty)\times\mathbb{R}), \text{ } \\
\text{ii)}& \textbf{ }|f(x,y)|\leq Var(f(x,\cdot),\mathbb{R})\leq Var(f,(-\infty,x]\times\mathbb{R}), \text{ } \\
\text{iii)}&\text{ } |f(x,y)|\leq Var(f(\cdot,y),\mathbb{R})\leq Var(f,\mathbb{R}\times[y,\infty)), \text{ } \\
\text{iv)}&\text{ } |f(x,y)|\leq Var(f(\cdot,y),\mathbb{R})\leq Var(f,\mathbb{R}\times(-\infty,y]). \text{ }
\end{aligned}&&
\end{flalign*}
\end{lemma}

The space $BV_{H_0}(\mathbb{R}^2)$ is defined in \cite{Mendoza3} as 
\begin{equation*}
BV_{H_0}(\mathbb{R}^2)=\{f\in BV_{H}(\mathbb{R}^2):\text{ }\lim \limits
_{\substack{ a \to \pm\infty}}f(a,y)=0=\lim \limits_{\substack{ b \to
\pm\infty}}f(x,b),\text{ } \forall x,y\in\mathbb{R}\}.
\end{equation*}
From Lemma \ref{TA1} and Observation \ref{Tlim}, we obtain the following
characterization theorem that confirms the equality (2.3) in \cite{Mendoza3}.

\begin{theorem}
\label{CTH} The function $f\in BV_{||0||}(\mathbb{R}^2)$ if and only if $%
f\in BV_{H_0}(\mathbb{R}^2)$.
\end{theorem}

%%%%%%%%%%%%%%%

It is important to note that 
\begin{equation}  \label{ENOC}
\begin{split}
L^1(\mathbb{R}^2)\cap BV_{H}(\mathbb{R}^2)\subsetneq BV_{||0||}(\mathbb{R}%
^2) \not\subset L^1(\mathbb{R}^2).
\end{split}%
\end{equation}

\begin{example}
{}\label{EFBV}  The function 
\begin{equation*}  \label{EFBVH0}
f(x,y) = 
\begin{cases}
(1/x)(1/y) & \quad\text{ si } x,y\geq 1 \\ 
0 & \quad\text{ si } x<1 \text{ o } y<1,%
\end{cases}%
\end{equation*}
belongs to $BV_{||0||}(\mathbb{R}^2)\setminus L^1(\mathbb{R}^2)$. This
function illustrates the contention relationships in (\ref{ENOC}).
\end{example}

\begin{proposition}
\label{TY4} If $g\in C_b(\mathbb{R}^2)$ and $f\in BV_{||0||}(\mathbb{R}^2)$,
then 
\begin{equation*}
\iint_{\mathbb{R}^2}g(t_1,t_2)df(t_1,t_2)=\iint_{\mathbb{R}%
^2}f(t_1,t_2)dg(t_1,t_2).
\end{equation*}
\end{proposition}

\begin{proof}
Let $R=[a,b]\times[c,d]\subset\mathbb{R}^2$ be. By Theorem \ref{CTH} and the
Integration by Parts Theorem in \cite{Moricz2}, we have that
\begin{equation}
\begin{split}
\iint_{R}g(t_1,t_2)df(t_1,t_2) =&\text{ } f(b,d)g(b,d)-f(b,c)g(b,c) \\
&-f(a,d)g(a,d)+f(a,c)g(a,c) \\
&-\int^b_a g(t_1,d)df(t_1,d)+\int^b_a g(t_1,c)df(t_1,c) \\
&-\int^d_c g(b,t_2)df(b,t_2)+\int^d_c g(a,t_2)df(a,t_2) \\
&+\iint_{R}f(t_1,t_2)dg(t_1,t_2).  \label{EA31}
\end{split}%
\end{equation}

It is immediate 
\begin{equation*}
\begin{split}
\lim \limits_{\substack{ a,c \to -\infty \\ b,d \to \infty}}f(b,d)g(b,d)=&%
\text{ } \lim \limits_{\substack{ a,c \to -\infty \\ b,d \to \infty}}%
f(b,c)g(b,c) \\
=\lim \limits_{\substack{ a,c \to -\infty \\ b,d \to \infty}}f(a,d)g(a,d)%
\text{ }&= \lim \limits_{\substack{ a,c \to -\infty \\ b,d \to \infty}}%
f(a,c)g(a,c)=0.
\end{split}%
\end{equation*}

In addition, according to Lemma \ref{TA1}, it is satisfied that 
\begin{equation*}
\begin{split}
\left|\int^b_a g(t_1,d)df(t_1,d)\right|\leq&\text{ }Var(f,\mathbb{R}\times
[d,\infty))||g||_\infty.
\end{split}%
\end{equation*}
By Observation \ref{Tlim}, it follows that 
\begin{equation*}  \label{EA231}
\lim \limits_{\substack{ a,c \to -\infty \\ b,d \to \infty}}\int^b_a
g(t_1,d)df(t_1,d)=0.
\end{equation*}
Similarly, it shows that 
\begin{equation*}
\begin{split}
\lim \limits_{\substack{ a,c \to -\infty \\ b,d \to \infty}}\int^b_a
g(t_1,c)df(t_1,c)=&\text{ } \lim \limits_{\substack{ a,c \to -\infty \\ b,d
\to \infty}}\int^d_c g(b,t_2)df(b,t_2) \\
=&\lim \limits_{\substack{ a,c \to -\infty \\ b,d \to \infty}}\int^d_c
g(a,t_2)df(a,t_2)=\text{ } 0.
\end{split}%
\end{equation*}

By \cite[Lemma 3.5]{Ghodadra}, the improper Riemann-Stieltjes integral of $g$
with respect to $f$ exists. Therefore, applying the limit in (\ref{EA31}) as 
$a,c \to -\infty$ and $b,d \to -\infty$ we conclude that 
\begin{equation*}
\iint_{\mathbb{R}^2}g(t_1,t_2)df(t_1,t_2)=\iint_{\mathbb{R}%
^2}f(t_1,t_2)dg(t_1,t_2).
\end{equation*}
\end{proof}

\begin{proposition}
\label{TY6} Let $f\in BV_{||0||}(\mathbb{R}^2)$ be. Then, for $u_1<u_2$ and $%
v_1<v_2$, 
\begin{equation*}
\begin{split}
\lim \limits_{\substack{ a,c \to -\infty \\ b,d \to \infty}} \iint_{[a,b]%
\times[c,d]}f(t_1,t_2)\left(\int^{u_2}_{u_{1}}cos(t_1\tau)d\tau\right)
\left(\int^{v_2}_{v_{1}}cos(t_2\tau)d\tau\right)d(t_1,t_2) \\
= \iint_{\mathbb{R}^2}f(t_1,t_2)d\left(\int^{u_2}_{u_{1}}\frac{sen(t_1\tau)}{%
\tau}d\tau\right)\left(\int^{v_2}_{v_{1}}\frac{sen(t_2\tau)}{\tau}%
d\tau\right).
\end{split}%
\end{equation*}
\end{proposition}

\begin{proof}
Let $R=[a,b]\times[c,d]\subset\mathbb{R}^2$ be. By \cite[Lemma 4]{Moricz1}
we define, for each $(t_1,t_2)\in \mathbb{R}^2$, the functions 
\begin{equation*}
\begin{split}
g(t_1,t_2):=&\left(\int^{u_2}_{u_{1}}cos(t_1\tau)d\tau\right)
\left(\int^{v_2}_{v_{1}}cos(t_2\tau)d\tau\right)\in C(\mathbb{R}^2)
\end{split}%
\end{equation*}
and 
\begin{equation*}
\begin{split}
h(t_1,t_2):=&\left(\int^{u_2}_{u_{1}}\frac{sen(t_1\tau)}{\tau}d\tau\right)
\left(\int^{v_2}_{v_{1}}\frac{sen(t_2\tau)}{\tau}d\tau\right)\in C_0(\mathbb{%
R}^2),
\end{split}%
\end{equation*}
which satisfy the relation 
\begin{equation}
\begin{split}
\int^{t_1}_a\int^{t_2}_c g(x,y)dxdy=&h(t_1,t_2)-h(t_1,c)-h(a,t_2)+h(a,c).
\end{split}
\label{EA41}
\end{equation}

Since $h\in C(R)$ and $f\in BV_H(R)$, applying the Integration by Parts Theorem 
\cite{Moricz2}, we have that $\iint_Rfdh$ exists. Similary, we prove that the following integrals 
\begin{equation}
\begin{split}
\iint_{R}f(t_1,t_2)dh(t_1,c), \text{ } \iint_{R}f(t_1,t_2)dh(a,t_2) \text{
and } \iint_{R}f(t_1,t_2)dh(a,c)  \label{EA42}
\end{split}%
\end{equation}
exist and are equal to zero.\newline

From (\ref{EA41}) and Lemma \ref{IIIR}, we have that 
\begin{equation}
\begin{split}
\iint_{R}f(t_1,t_2)dh(t_1,t_2)-\iint_{R}f(t_1,t_2)dh(t_1,c)-%
\iint_{R}f(t_1,t_2)dh(a,t_2) \\
+\iint_{R}f(t_1,t_2)dh(a,c) =\iint_{R}f(t_1,t_2)g(t_1,t_2)d(t_1,t_2).
\label{EA43}
\end{split}%
\end{equation}

According to (\ref{EA42}) and (\ref{EA43}),
\begin{equation*}
\begin{split}
\iint_{R}f(t_1,t_2)g(t_1,t_2)d(t_1,t_2)=\iint_{R}f(t_1,t_2)dh(t_1,t_2).
\end{split}%
\end{equation*}

By Proposition \ref{TY4}, $\iint_{\mathbb{R}^2}fdh$ exists, then applying
the limit to the above equality as $a,c \to -\infty$ and $b,d \to \infty$ we
conclude that 
\begin{equation*}
\begin{split}
\lim \limits_{\substack{ a,c \to -\infty \\ b,d \to \infty}} \iint_{[a,b]%
\times[c,d]}f(t_1,t_2)g(t_1,t_2)d(t_1,t_2)=\iint_{\mathbb{R}%
^2}f(t_1,t_2)dh(t_1,t_2).
\end{split}%
\end{equation*}
\end{proof}

\subsection{Sequences in $\left( L \right)$}

The following definition can be found in \cite{Bary} and \cite{Moricz1}.

\begin{definition}
\label{Lacu} An increasing sequence $\left\{ u_{j}\right\}_{j\in\mathbb{N}}
\subset \mathbb{R}^{+}$ is said to satisfy \textbf{Lacunary's condition (L)}
and is denoted as $\left\{ u_{j}\right\}_{j\in\mathbb{N}} \in (L),$ if there
exists $A>1$ such that 
\begin{equation*}
u_{m}\sum_{j=m}^{\infty }\frac{1}{u_{j}}\leq A,\text{ \ \ }m=1,2,3,....
\label{4}
\end{equation*}
\end{definition}

\begin{example}
The sequence $\{2^{j}\}_{j\in \mathbb{N}}$ satisfies condition $(L)$ with $A=2$.
\end{example}

Some results of our interest associated with sequences in $(L)$ are the
following.

\begin{lemma}[{\textbf{\protect\cite[Lemma 2]{Moricz1}}}]
\label{LM1} If $\left\{ u_{j}\right\}_{j\in\mathbb{N}} \in \left( L\right) ,$
then 
\begin{equation*}
\sum_{j=1}^{\infty }\max_{u_{j-1}\leq v\leq u_{j}}\left\vert
\int_{u_{j-1}}^{v}\frac{sen (tu)}{u}du\right\vert \leq 3A+4,\text{ \ \ }%
t\neq 0,
\end{equation*}%
where $u_{0}:=0$ and $A$ is derived from Definition \ref{Lacu}.
\end{lemma}

\begin{lemma}[{\textbf{\protect\cite[Lema 3]{Moricz1}}}]
\label{LM2} If $\left\{ u_{j}\right\}_{j\in\mathbb{N}} \in \left( L\right) ,$
then 
\begin{equation*}
\sum_{j=m+1}^{\infty }\max_{u_{j-1}\leq v\leq u_{j}}\left\vert
\int_{u_{j-1}}^{v}\frac{sen (tu)}{u}du\right\vert \leq \frac{3A}{\left\vert
t\right\vert u_{m}},\text{ \ \ }m\in\mathbb{N};\text{ }t\neq 0.
\end{equation*}
\end{lemma}

\section{KP-Fourier Transform}

In this section we define a transform that generalizes the Fourier transform
operator since it can be applied to non-absolutely integrable functions. For
this, we begin defining the Kurzweil-Henstock integral over a bounded
rectangle $R=[a,b]\times[c,d]$, \cite{Lee}, \cite{Muldowney}.

\begin{itemize}
\item[\textit{i)}] Let $\delta :R\longrightarrow \mathbb{R}$ be a function. It is
said that $\delta$ is a \textbf{gauge} on $R$, if $\delta(t_1,t_2)\geq0$ for
all $(t_1,t_2)\in R$.

\item[\textit{ii)}] Let $P=\{R_{i,j}\}\in \mathcal{P}(R)$ (defined in Section 2) and $(\xi_{i,j},\eta_{i,j})%
\in R_{i,j}$. It is said that $P$ is \textbf{$\delta-$fine}, if $%
[x_{i-1},x_{i}]\subset
(\xi_{i,j}-\delta(\xi_{i,j},\eta_{i,j}),\xi_{i,j}+\delta(\xi_{i,j},%
\eta_{i,j}))$ and $[y_{j-1},y_{j}]\subset
(\eta_{i,j}-\delta(\xi_{i,j},\eta_{i,j}),\eta_{i,j}+\delta(\xi_{i,j},%
\eta_{i,j}))$.

\item[\textit{iii)}] Let $P=\{R_{i,j}\}\in \mathcal{P}(R)$ and $(\xi_{i,j},%
\eta_{i,j})\in R_{i,j}$. Given a function $f:R\longrightarrow \mathbb{R}$,
the \textbf{\ Riemann sum of $f$ over} $P$ is defined as 
\begin{equation*}
S(f;P)=\sum_{i,j}f(\xi_{i,j},\eta_{i,j})(x_{i}-x_{i-1})(y_{j}-y_{j-1}).
\end{equation*}
\end{itemize}

\begin{definition}
A function $f:R\rightarrow \mathbb{R}$ is said to be
\textbf{Kurzweil-Henstock (KH) integrable over} $R$ and we denote it by $f\in
KH(R)$, if there exists $A\in\mathbb{R}$ such that for $\varepsilon >0$,
there exists a gauge $\delta_{\varepsilon}$ on $R$ such that for $%
P=\{R_{i,j}\}\in \mathcal{P}(R)$ which is $\delta_{\varepsilon}-$fine, then 
\begin{equation*}
|S(f;P)-A|<\varepsilon .
\end{equation*}%
Moreover, $A$ is the Kurzweil-Henstock integral of $f$ and is denoted by $$%
A=\iint_{R}f(t_1,t_2)d(t_1,t_2).$$
\end{definition}

The space of KH integrable functions over compact rectangles is denoted by $%
KH_{loc}(\mathbb{R}^2)$. In \cite{Lee}, it is proved that the space of
locally Lebesgue integrable functions $L^1_{loc}(\mathbb{R}^2)$ is properly contained
in $KH_{loc}(\mathbb{R}^2)$.

\begin{definition}
\cite[Definition 4.1]{Mendoza3}\label{Definition P-Fourier} Let $f:\mathbb{R}%
^2\rightarrow\mathbb{R}$ be a function such that $f(\cdot)e^{-i<\cdot,v>}\in
KH_{loc}(\mathbb{R}^2)$. The \textbf{KP-Fourier Transform} of $f$ at $%
(\xi,\eta)\in\mathbb{R}^2$ is defined as 
\begin{equation}
\mathcal{F}(f) (\xi,\eta):=\lim \limits_{\substack{ a,c \to
-\infty \\ b,d \to \infty}} \iint_{[a,b]\times[c,d]}f(t_1,t_2)e^{-i(\xi
t_1+\eta t_2)}d(t_1,t_2),  \label{DTFI}
\end{equation}
when the limit in (\ref{DTFI}) exists in the Pringsheim sense.
\end{definition}

\begin{observation}
Let us observe that if $f\in L^1(\mathbb{R}^2)$, then the KP-Fourier
transform $\mathcal{F}(f)$ is well-defined for each $(\xi, \eta)\in\mathbb{R}%
^2$ and is equal to the Fourier transform $\hat{f}$.
\end{observation}

The following result shows that the KP-Fourier transform operator is well
defined on $BV_{||0||}(\mathbb{R}^2)$ which is not contained in $L^1(\mathbb{%
R}^2)$ according to (\ref{ENOC}).

\begin{theorem}
\label{TY2} Suppose that $f\in BV_{||0||}(\mathbb{R}^2)$. Then, for $\xi\neq0
$ and $\eta\neq0$, 
\begin{equation*}
\mathcal{F}(f) (\xi,\eta)=-\frac{1}{\xi\eta}\iint_{\mathbb{R}%
^2}e^{-i(\xi t_1+\eta t_2)}df(t_1,t_2).
\end{equation*}
\end{theorem}

\begin{proof}

Let $R=[a,b]\times[c,d]\subset\mathbb{R}^2$ and $\xi\neq0$, $\eta\neq0$. Let 
$g(t_1,t_2)=e^{-i(\xi t_1+\eta t_2)}$ and $G(t_1,t_2)=-e^{-i(\xi t_1+\eta
t_2)}/\xi\eta$, for each $(t_1,t_2)\in\mathbb{R}^2$. By \cite[Theorem 6.5.9]%
{Lee}, we have that $fg\in KH(R)$ and 
\begin{equation}
\begin{split}
\iint_{R}f(t_1,t_2)g(t_1,t_2)d(t_1,t_2) =&
f(b,d)G(b,d)-f(b,c)G(b,c) \\
-&f(a,d)G(a,d)+f(a,c)G(a,c) \\
-&\int^b_a G(t_1,d)df(t_1,d)+\int^b_a G(t_1,c)df(t_1,c) \\
-&\int^d_c G(b,t_2)df(b,t_2)+\int^d_c G(a,t_2)df(a,t_2) \\
+&\iint_{R}G(t_1,t_2)df(t_1,t_2).  \label{ET21}
\end{split}%
\end{equation}

It is easy to prove that 
\begin{equation*}
\begin{split}
\lim \limits_{\substack{ a,c \to -\infty \\ b,d \to \infty}}f(b,d)G(b,d)=&
\lim \limits_{\substack{ a,c \to -\infty \\ b,d \to \infty}}f(b,c)G(b,c), \\
=\lim \limits_{\substack{ a,c \to -\infty \\ b,d \to \infty}}f(a,d)G(a,d)=&
\lim \limits_{\substack{ a,c \to -\infty \\ b,d \to \infty}}f(a,c)G(a,c)=0.
\end{split}%
\end{equation*}

From Lemma \ref{TA1}, 
\begin{equation*}
\begin{split}
\left|\int^b_a G(t_1,d)df(t_1,d)\right|\leq&\frac{2}{|\xi\eta|}Var(f,\mathbb{%
R}\times [d,\infty)).
\end{split}%
\end{equation*}
By Observation \ref{Tlim}, it follows that 
\begin{equation*}
\lim \limits_{\substack{ a,c \to -\infty \\ b,d \to \infty}}\int^b_a
G(t_1,d)df(t_1,d)=0.
\end{equation*}
Similarly, we have that 
\begin{equation*}
\begin{split}
\lim \limits_{\substack{ a,c \to -\infty \\ b,d \to \infty}}\int^b_a
G(t_1,c)df(t_1,c)=& \lim \limits_{\substack{ a,c \to -\infty \\ b,d \to
\infty}}\int^d_c G(b,t_2)df(b,t_2) \\
=&\lim \limits_{\substack{ a,c \to -\infty \\ b,d \to \infty}}\int^d_c
G(a,t_2)df(a,t_2)=\text{ } 0.  \label{EA222}
\end{split}%
\end{equation*}

Since the real and imaginary part of $G$ belong to $C_b(\mathbb{R}^2)$,
according to Proposition \ref{TY4}, we can claim that $\iint_{\mathbb{R}%
^2}Gdf$ exists. Therefore, applying the limit to (\ref{ET21}) as $%
a,c\rightarrow$ and $b,d\rightarrow$ we concluded that 
\begin{equation*}
\begin{split}
\lim \limits_{\substack{ a,c \to -\infty \\ b,d \to \infty}}\iint_R
f(t_1,t_2)e^{-i(\xi t_1+\eta t_2)}d(t_1,t_2)= -\frac{1}{\xi\eta}\iint_{%
\mathbb{R}^2}e^{-i(\xi t_1+\eta t_2)}df(t_1,t_2).
\end{split}%
\end{equation*}
\end{proof}

\begin{example}
Considering the function defined in the Example \ref{EFBV} and by Theorem %
\ref{TY2} we prove that, for $(\xi,\eta)\in\mathbb{R}^2$ with $\xi\neq 0$
and $\eta\neq 0$, its KP-Fourier transform is 
\begin{equation*}
\mathcal{F}(f)(\xi,\eta)=\Gamma(0,i\xi)\Gamma(0,i\eta),
\end{equation*}
where $\Gamma(\cdot,\cdot)$ is the incomplete Gamma function. Let us note
that when $\xi=0$ or $\eta=0$, $\mathcal{F}(f)(\xi,\eta)$ does not exist.
\end{example}

In \cite[Corollary 4.1]{Mendoza3} the following result is proved, however we
provide an alternative proof.

\begin{proposition}
\label{PA5} Let $f\in BV_{||0||}(\mathbb{R}^2)$ be. Then $\mathcal{F}(f)$ is
continuous at $(\xi,\eta)$ with $\xi\neq0$ and $\eta\neq0$.
\end{proposition}

\begin{proof}
By Theorem \ref{TY2}, for $f\in BV_{||0||}(\mathbb{R}^2)$ and $%
(\xi_0,\eta_0)\in\mathbb{R}^2$ with $\xi_0\neq 0 \text{ and }
\eta_0\neq 0$, it is satisfied that 
\begin{equation}  \label{TTC1}
\begin{split}
\mathcal{F}(f) (\xi_0,\eta_0)=&\text{ }-\frac{1}{\xi_0\eta_0}\iint_{%
\mathbb{R}^2}cos(\xi_0 t_1+\eta_0 t_2)df(t_1,t_2) \\
&+\frac{i}{\xi_0\eta_0}\iint_{\mathbb{R}^2}sin(\xi_0 t_1+\eta_0
t_2)df(t_1,t_2).
\end{split}%
\end{equation}

Let $(\xi,\eta)\in \{(x,y)\in\mathbb{R}^2:\text{ }x\neq 0 \text{ and }
y\neq 0\}$. By Lemma \ref{TG1} and the Mean Value Theorem, we obtain the
following inequality 
\begin{multline*}
\left|\iint_{[a,b]\times[c,d]}cos(\xi t_1+\eta t_2)df(t_1,t_2)-\iint_{[a,b]%
\times[c,d]}cos(\xi_0 t_1+\eta_0 t_2)df(t_1,t_2)\right|
\end{multline*}
\begin{equation*}
\begin{split}
\leq& \text{ }|\xi-\xi_0|\iint_{[a,b]\times[c,d]}|t_1|dV(f;t_1,t_2)+|\eta-%
\eta_0|\iint_{[a,b]\times[c,d]}|t_2|dV(f;t_1,t_2).
\end{split}%
\end{equation*}
Then, for each $[a,b]\times[c,d]\subset\mathbb{R}^2,$ 
\begin{multline}  \label{TTC4}
\lim \limits_{\substack{ (\xi,\eta) \to (\xi_0,\eta_0)}}\iint_{[a,b]\times[%
c,d]}cos(\xi t_1+\eta t_2)df(t_1,t_2) \\
=\iint_{[a,b]\times[c,d]}cos(\xi_0 t_1+\eta_0 t_2)df(t_1,t_2).
\end{multline}

By Observation \ref{Tlim}, given $\varepsilon>0$ there exists $M>0$ such
that 
\begin{equation*}
Var(f,[M,\infty)\times \mathbb{R})<\varepsilon/4,
\end{equation*}
\begin{equation*}
Var(f, \mathbb{R}\times[M,\infty))<\varepsilon/4,
\end{equation*}
\begin{equation*}
Var(f,(-\infty,-M]\times \mathbb{R})<\varepsilon/4 
\end{equation*}
\begin{equation*}
\text{and }Var(f,\mathbb{R}\times(-\infty,-M])<\varepsilon/4.
\end{equation*}

Suppose that $[a,b]\times[c,d],[a_1,b_1]\times[c_1,d_1]\supset [-M,M]^2$.
Applying $ii)$ of Theorem \ref{TG1} on compact rectangles and the previous
inequalities we have that, for each $(\xi,\eta)\in \mathbb{R}^2$, 
\begin{multline*}
\left|\iint_{[a_1,b_1]\times[c_1,d_1]}cos(\xi t_1+\eta
t_2)df(t_1,t_2)-\iint_{[a,b]\times[c,d]}cos(\xi t_1+\eta
t_2)df(t_1,t_2)\right| \\
< \varepsilon.
\end{multline*}
Thus, the limit 
\begin{equation}  \label{TTC5}
\lim \limits_{\substack{ a,c \to -\infty \\ b,d \to \infty}}\iint_{[a,b]%
\times[c,d]}cos(\xi t_1+\eta t_2)df(t_1,t_2)
\end{equation}
is uniform with respect to $(\xi,\eta)\in\mathbb{R}^2$.\newline
Applying \cite[Theorem 1]{Kadelburg}, (\ref{TTC4}) and (\ref{TTC5}),
\begin{multline}  \label{TTC2}
\lim \limits_{\substack{ (\xi,\eta) \to (\xi_0,\eta_0)}}-\frac{1}{%
\xi\eta}\iint_{\mathbb{R}^2}cos(\xi t_1+\eta t_2)df(t_1,t_2) \\
=-\frac{1}{\xi_0\eta_0}\iint_{\mathbb{R}^2}cos(\xi_0 t_1+\eta_0
t_2)df(t_1,t_2).
\end{multline}

Similarly, it is proved that 
\begin{multline}  \label{TTC3}
\lim \limits_{\substack{ (\xi,\eta) \to (\xi_0,\eta_0)}}\frac{i}{%
\xi\eta}\iint_{\mathbb{R}^2}sin(\xi t_1+\eta t_2)df(t_1,t_2) \\
=\frac{i}{\xi_0\eta_0}\iint_{\mathbb{R}^2}sin(\xi_0 t_1+\eta_0
t_2)df(t_1,t_2).
\end{multline}

According to (\ref{TTC1}), (\ref{TTC2}) and (\ref{TTC3}), we conclude that $%
\mathcal{F}(f)$ is continuous at $(\xi_0,\eta_0)$.
\end{proof}

For $0<\alpha_i<\beta_i<\infty$ with $i=1,2$, we denote 
\begin{equation*}
R_{\alpha_1,\alpha_2}^{\beta_1,\beta_2}=\{(x,y)\in\mathbb{R}%
^2:\alpha_1\leq|x|\leq \beta_1, \text{ }\alpha_2\leq|y|\leq \beta_2\}
\end{equation*}
and 
\begin{equation*}
h_{\alpha_i,\beta_i}(t)=(sin(\beta_i t)-sin(\alpha_{i}t))/{\pi t}.
\end{equation*}

\begin{proposition}
\label{TY3} Suppose that $f\in BV_{||0||}(\mathbb{R}^2)$ and $(x,y)\in%
\mathbb{R}^2$. Then 
\begin{equation*}
\begin{split}
\frac{1}{4\pi^2}\underset{R_{\alpha_1,\alpha_2}^{\beta_1,\beta_2}}{\iint }\mathcal{F}%
(f)(\varepsilon,\eta)e^{i(x\varepsilon+y\eta)}d(\varepsilon,\eta)=\frac{1}{%
\pi^2}\lim \limits_{\substack{ a,c \to -\infty \\ b,d \to \infty}}
\iint_{[a,b]\times[c,d]}f(x+t_1,y+t_2) \\
\times\left(\frac{sin(\beta_1t_1)-sin(\alpha_1t_1)}{t_1}\right)\left(\frac{%
sin(\beta_2t_2)-sin(\alpha_2t_2)}{t_2}\right)d(t_1,t_2).
\end{split}%
\end{equation*}
\end{proposition}

\begin{proof}
Theorem \ref{TY2} and Proposition \ref{PA5} give us conditions to apply
Lebesgue's Dominated Convergence Theorem and Fubini's Theorem in the
following equalities 
\begin{multline*}
\frac{1}{4\pi^2}\underset{R_{\alpha_1,\alpha_2}^{\beta_1,\beta_2}}{\iint }\mathcal{F}%
(f)(\varepsilon,\eta)e^{i(x\varepsilon+y\eta)}d(\varepsilon,\eta) \\
=\frac{1}{4\pi^2}\lim\limits_{\substack{ a,c \to -\infty \\ b,d \to \infty}}%
\underset{R_{\alpha_1,\alpha_2}^{\beta_1,\beta_2}}{\iint}\iint_{[a,b]\times[%
c,d]} f(t_1,t_2)
e^{-i(t_1\varepsilon+t_2\eta)}d(t_1,t_2)e^{i(x\varepsilon+y\eta)}d(%
\varepsilon,\eta) \\
=\text{ }\frac{1}{4\pi^2}\lim\limits_{\substack{ a,c \to -\infty \\ b,d \to
\infty}}\iint_{[a,b]\times[c,d]}\underset{R_{\alpha_1,\alpha_2}^{\beta_1,%
\beta_2}}{\iint} f(t_1,t_2) e^{-i(t_1-x)\varepsilon}e^{-i(t_2-y)\eta}
d(\varepsilon,\eta) d(t_1,t_2) \\
= \text{ }\lim\limits_{\substack{ a,c \to -\infty \\ b,d \to
\infty}}\iint_{[a,b]\times[c,d]} f(t_1,t_2)h_{\alpha_1,\beta_1}(x-t_1)
h_{\alpha_2,\beta_2}(y-t_2) d(t_1,t_2)
\end{multline*}
\begin{multline}
= \text{ }\lim\limits_{\substack{ a,c \to -\infty \\ b,d \to
\infty}}\iint_{[a,b]\times[c,d]} f(x+\mu,y+\tau)h_{\alpha_1,\beta_1}(\mu)
h_{\alpha_2,\beta_2}(\tau) d(\mu,\tau).  \label{EQTY3}
\end{multline}

The last equality is obtained by making the change of variable $\mu=t_1-x$
and $\tau=t_2-y$.
\end{proof}

\begin{proposition}
\label{PAI} Let $f\in BV_{||0||}(\mathbb{R}^2)$ and $(x,y)\in\mathbb{R}^2$.
Then the function 
\begin{equation*}
g_{(x,y)}(t_1,t_2)=f(x-t_1,y-t_2)+f(x-t_1,y+t_2)+f(x+t_1,y-t_2)+f(x+t_1,y+t_2)
\end{equation*}
belongs to $BV_{||0||}(\mathbb{R}^2)$ and the limit 
\begin{equation}  \label{EUNI}
\lim_{a \rightarrow \infty} \underset{[0,a]\times[0,a]}{\iint }%
g_{(x,y)}(t_1,t_2)h_{\alpha_1,\beta_1}(t_1)h_{\alpha_2,%
\beta_2}(t_2)d(t_1,t_2)
\end{equation}
exists and is uniform with respect to $0<\alpha_i<\beta_i<\infty$ for $i=1,2$%
.
\end{proposition}

\begin{proof}
Making the change of variable $\mu=x-t_1$ and $\tau=y-t_2$ in (\ref{EQTY3})
from Proposition \ref{TY3} we have that, for $0<\alpha_i<\beta_i<\infty$%
, $i=1,2$, the limit 
\begin{equation*}
\begin{split}
\lim_{a \rightarrow \infty} \underset{[-a,a]\times[-a,a]}{\iint }%
f(x-t_1,y-t_2)h_{\alpha_1,\beta_1}(t_1)h_{\alpha_2,\beta_2}(t_2)d(t_1,t_2)
\end{split}%
\end{equation*}
exists.

Considering the appropriate change of variables and for $0<\alpha_i<\beta_i<\infty$, $i=1,2$, 
\begin{equation}  \label{ERalpha}
\begin{split}
\lim_{a \rightarrow \infty} \underset{[-a,a]\times[-a,a]}{\iint }%
f(x-t_1,y-t_2)h_{\alpha_1,\beta_1}(t_1)h_{\alpha_2,\beta_2}(t_2)d(t_1,t_2) \\
=\lim_{a \rightarrow \infty} \underset{[0,a]\times[0,a]}{\iint }%
g_{(x,y)}(t_1,t_2)h_{\alpha_1,\beta_1}(t_1)h_{\alpha_2,%
\beta_2}(t_2)d(t_1,t_2).
\end{split}%
\end{equation}

Now, let us prove that the limit in (\ref{EUNI}) is uniform. Since $g_{(x,y)}$
belongs to $BV_{||0||}(\mathbb{R}^2)$, given $\varepsilon>0$ there exists $%
\delta_0>0$ such that

\begin{itemize}
\item[\textit{i)}] If $(t_1,t_2)\in(0,\infty)^2$ with $||(t_1,t_2)||>\delta_0$,
then $|g_{(x,y)}(t_1,t_2)|<\varepsilon$, 

\item[\textit{ii)}] $Var(g_{(x,y)},[\delta_0,\infty)\times[0,\infty))<\varepsilon$, 

\item[\textit{iii)}] $Var(g_{(x,y)},[0,\infty)\times[\delta_0,\infty))<\varepsilon$.
\end{itemize}

For each $(z_1,z_2)\in\mathbb{R}^2$ and $0<\alpha_i<\beta_i<\infty$, $i=1,2$%
, we define the function 
\begin{equation*}
H_{\alpha_1,\alpha_2}^{\beta1,\beta_2}(z_1,z_2)=\int^{z_1}_0\int^{z_2}_0
h_{\alpha_1,\beta_1}(t_1)h_{\alpha_2,\beta_2}(t_2)dt_1dt_2.
\end{equation*}

Let $a_2>a_1\geq\delta_0$. Applying the Integration by Parts Theorem \cite[Theorem
6.5.9]{Lee} and the above statements, we obtain that 
\begin{multline*}
\left| \underset{[0,a_1]\times[0,a_1]}{\iint }g_{(x,y)}(t_1,t_2)h_{\alpha_1,%
\beta_1}(t_1)h_{\alpha_2,\beta_2}(t_2)d(t_1,t_2)\right. \\
\left.-\underset{[0,a_2]\times[0,a_2]}{\iint }g_{(x,y)}(t_1,t_2)h_{\alpha_1,%
\beta_1}(t_1)h_{\alpha_2,\beta_2}(t_2)d(t_1,t_2) \right|
\end{multline*}
\begin{multline*}
\leq\text{ } \left|\underset{[0,a_2]\times[a_1,a_2]}{\iint }%
g_{(x,y)}(t_1,t_2)h_{\alpha_1,\beta_1}(t_1)h_{\alpha_2,%
\beta_2}(t_2)d(t_1,t_2)\right| \\
+\left|\underset{[a_1,a_2]\times[0,a_1]}{\iint }g_{(x,y)}(t_1,t_2)h_{%
\alpha_1,\beta_1}(t_1)h_{\alpha_2,\beta_2}(t_2)d(t_1,t_2)\right|
\end{multline*}
\begin{equation*}
\begin{split}
& =\text{ }\left|H_{\alpha_1,\alpha_2}^{\beta1,%
\beta_2}(a_2,a_2)g_{(x,y)}(a_2,a_2)-H_{\alpha_1,\alpha_2}^{\beta1,%
\beta_2}(0,a_2)g_{(x,y)}(0,a_2)+H_{\alpha_1,\alpha_2}^{\beta1,%
\beta_2}(0,a_1)g_{(x,y)}(0,a_1)\right. \\
&
-H_{\alpha_1,\alpha_2}^{\beta1,\beta_2}(a_2,a_1)g_{(x,y)}(a_2,a_1)-%
\int^{a_2}_0
H_{\alpha_1,\alpha_2}^{\beta1,\beta_2}(\cdot,a_2)dg_{(x,y)}(\cdot,a_2)+%
\int^{a_2}_0
H_{\alpha_1,\alpha_2}^{\beta1,\beta_2}(\cdot,a_1)dg_{(x,y)}(\cdot,a_1) \\
& \left.-\int^{a_2}_{a_1}
H_{\alpha_1,\alpha_2}^{\beta1,\beta_2}(a_2,\cdot)dg_{(x,y)}(a_2,\cdot)+%
\int^{a_2}_{a_1}
H_{\alpha_1,\alpha_2}^{\beta1,\beta_2}(0,\cdot)dg_{(x,y)}(0,\cdot)+%
\iint_{[0,a_2]\times[a_1,a_2]}H_{\alpha_1,\alpha_2}^{\beta1,\beta_2}dg_{(x,y)}\right| \\
&
+\left|H_{\alpha_1,\alpha_2}^{\beta1,\beta_2}(a_2,a_1)g_{(x,y)}(a_2,a_1)-H_{%
\alpha_1,\alpha_2}^{\beta1,\beta_2}(a_1,a_1)g_{(x,y)}(a_1,a_1)+H_{\alpha_1,%
\alpha_2}^{\beta1,\beta_2}(a_2,0)g_{(x,y)}(a_2,0)\right. \\
&
-H_{\alpha_1,\alpha_2}^{\beta1,\beta_2}(a_1,0)g_{(x,y)}(a_1,0)-%
\int^{a_2}_{a_1}
H_{\alpha_1,\alpha_2}^{\beta1,\beta_2}(\cdot,a_1)dg_{(x,y)}(\cdot,a_1)+%
\int^{a_2}_{a_1}
H_{\alpha_1,\alpha_2}^{\beta1,\beta_2}(\cdot,0)dg_{(x,y)}(\cdot,0) \\
& \left.-\int^{a_1}_{0}
H_{\alpha_1,\alpha_2}^{\beta1,\beta_2}(a_2,\cdot)dg_{(x,y)}(a_2,\cdot)+%
\int^{a_1}_{0}
H_{\alpha_1,\alpha_2}^{\beta1,\beta_2}(a_1,\cdot)dg_{(x,y)}(a_1,\cdot)+%
\iint_{[a_1,a_2]\times[0,a_1]}H_{\alpha_1,\alpha_2}^{\beta1,\beta_2}dg_{(x,y)}\right| \\
\end{split}%
\end{equation*}
\begin{equation*}
\begin{split}
\leq & \text{ } 2Si(\pi)^2\varepsilon+Si(\pi)^2Var(g_{(x,y)}(%
\cdot,a_2),[0,a_2])+Si(\pi)^2Var(g_{(x,y)}(\cdot,a_1),[0,a_2]) \\
&+Si(\pi)^2Var(g_{(x,y)}(a_2,\cdot),[a_1,a_2])+Si(\pi)^2Var(g_{(x,y)},[0,a_2]\times[a_1,a_2]) \\
&+2Si(\pi)^2\varepsilon+Si(\pi)^2Var(g_{(x,y)}(\cdot,a_1),[a_1,a_2])+Si(%
\pi)^2Var(g_{(x,y)}(a_2,\cdot),[0,a_1]) \\
&+Si(\pi)^2Var(g_{(x,y)}(a_1,\cdot),[0,a_1])+Si(\pi)^2Var(g_{(x,y)},[a_1,a_2]\times[0,a_1]) \\
<&\text{ } 12Si(\pi)^2\varepsilon,
\end{split}%
\end{equation*}
where $Si$ is the Sine Integral function defined for instance in \cite[%
Remark 2]{Mendoza2}.\newline

Therefore, given $\varepsilon>0$, there exists $\delta_0>0$ such that 
\begin{multline*}
\left| \underset{[0,a_1]\times[0,a_1]}{\iint }g_{(x,y)}(t_1,t_2)h_{\alpha_1,%
\beta_1}(t_1)h_{\alpha_2,\beta_2}(t_2)d(t_1,t_2)\right. \\
\left.-\underset{[0,a_2]\times[0,a_2]}{\iint }g_{(x,y)}(t_1,t_2)h_{\alpha_1,%
\beta_1}(t_1)h_{\alpha_2,\beta_2}(t_2)d(t_1,t_2)\right| <\text{ }
12Si(\pi)^2\varepsilon,
\end{multline*}
for $a_2,a_1\geq\delta_0$ and $0<\alpha_i<\beta_i<\infty$ with $i=1,2$.%
\newline
\end{proof}

\section{An extension of the Dirichlet-Jordan Theorem on $BV_{||0||}(\mathbb{R}^2)$.}

The first of our mains results is a two-dimensional extension of the Dirichlet-Jordan Theorem for functions in $BV_0(\mathbb{R})$.

\begin{theorem}\label{TDJI}
Suppose that $f\in BV_{||0||}(\mathbb{R}^2)$ and $(x,y)\in \mathbb{R}^2$.
Then, 
\begin{multline}\label{TIP}
\frac{1}{4\pi^2}\lim\limits_{\substack{ \alpha_1,\alpha_2 \to 0 \\ \beta_1,\beta_2 \to \infty
}}\iint_{R_{\alpha_1,\alpha_2}^{\beta_1,\beta_2}}\mathcal{F}%
(f)(\varepsilon,\eta)e^{i(x\varepsilon+y\eta)}d(\varepsilon,\eta) \\
=\frac{f(x+,y+)+f(x+,y-)+f(x-,y+)+f(x-,y-)}{4}.
\end{multline}
\end{theorem}

\begin{proof}
By Proposition \ref{TY3} and (\ref{ERalpha}), for $0<\alpha_i<\beta_i<%
\infty$ with $i=1,2$, 
\begin{equation}  \label{ERalpha2}
\begin{split}
\frac{1}{4\pi^2}\iint_{R_{\alpha_1,\alpha_2}^{\beta_1,\beta_2}}\mathcal{F}%
(f)(\varepsilon,\eta)e^{i(x\varepsilon+y\eta)}d(\varepsilon,\eta)=& \lim_{a
\rightarrow \infty} \underset{[0,a]\times[0,a]}{\iint }g_{(x,y)}(t_1,t_2) \\
&\times h_{%
\alpha_1,\beta_1}(t_1)h_{\alpha_2,\beta_2}(t_2)d(t_1,t_2).
\end{split}%
\end{equation}

We know that for $a>0$, the following statements are satisfied.

\begin{itemize}
\item[\textit{i)}] $\lim \limits_{\substack{ \alpha_2 \to 0 \\ \beta_2 \to \infty}}%
\int^a_0g_{(x,y)}(t_1,t_2)h_{\alpha_2,\beta_2}(t_2)dt_2=g_{(x,y)}(t_1,0+)/2$
for $t_1\in[0,a]$, 

\item[\textit{ii)}] $\left|\int^a_0g_{(x,y)}(t_1,t_2)h_{\alpha_2,\beta_2}(t_2)dt_2%
\right|\leq 2Si(\pi)\left[||g_{(x,y)}||_\infty+Var(g_{(x,y)},\mathbb{R}^2)%
\right]$ for $t_1\in[0,a]$ and $0<\alpha_2<\beta_2<\infty$, 

\item[\textit{iii)}] $\left\{\int^a_0g_{(x,y)}(\cdot,t_2)h_{\alpha_2,%
\beta_2}(t_2)dt_2:0<\alpha_2<\beta_2<\infty\right\}\subset L^1([0,a])$.
\end{itemize}

From Fubini's Theorem and Dominated Convergence Theorem we have, for $a>0$,
that 
\begin{equation}  \label{EPUNT}
\begin{split}
\lim\limits_{\substack{ \alpha_2 \to 0 \\ \beta_2 \to \infty}}\underset{[0,a]%
\times[0,a]}{\iint }g_{(x,y)}(t_1,t_2)h_{\alpha_1,\beta_1}(t_1)h_{\alpha_2,%
\beta_2}(t_2)d(t_1,t_2)= &\int^a_0 \frac{g_{(x,y)}(t_1,0+)}{2} \\
&\times h_{\alpha_1,\beta_1}(t_1)dt_1.
\end{split}%
\end{equation}
From Proposition \ref{PAI}, \cite[Theorem 1]{Kadelburg} and (\ref{EPUNT}),
it follows that 
\begin{multline}  \label{EINTER}
\lim\limits_{\substack{ \alpha_2 \to 0 \\ \beta_2 \to \infty}}%
\lim_{a\rightarrow\infty}\underset{[0,a]\times[0,a]}{\iint }%
g_{(x,y)}(t_1,t_2)h_{\alpha_1,\beta_1}(t_1)h_{\alpha_2,%
\beta_2}(t_2)d(t_1,t_2) \\
=\lim_{a\rightarrow\infty}\lim\limits_{\substack{ \alpha_2 \to 0 \\ \beta_2
\to \infty}}\underset{[0,a]\times[0,a]}{\iint }g_{(x,y)}(t_1,t_2)h_{%
\alpha_1,\beta_1}(t_1)h_{\alpha_2,\beta_2}(t_2)d(t_1,t_2).
\end{multline}
By (\ref{ERalpha2}) and (\ref{EINTER}), we obtain that 
\begin{equation*}
\begin{split}
\frac{1}{4\pi^2}\lim\limits_{\substack{ \alpha_1,\alpha_2 \to 0 \\ \beta_1,\beta_2 \to \infty
}}\underset{R_{\alpha_1,\alpha_2}^{\beta_1,\beta_2}}{\iint}\mathcal{F}%
(f)(\varepsilon,\eta)e^{i(x\varepsilon+y\eta)}d(\varepsilon,\eta)=&
\lim\limits_{\substack{ \alpha_1 \to 0 \\ \beta_1 \to \infty}}\lim\limits
_{\substack{ \alpha_2 \to 0 \\ \beta_2 \to \infty}}\lim_{a \rightarrow
\infty} \underset{[0,a]\times[0,a]}{\iint }g_{(x,y)}(t_1,t_2) \\
& \times h_{\alpha_1,\beta_1}(t_1)h_{\alpha_2,\beta_2}(t_2)d(t_1,t_2) \\
=&\lim\limits_{\substack{ \alpha_1 \to 0 \\ \beta_1 \to \infty}}%
\lim_{a\rightarrow\infty}\lim\limits_{\substack{ \alpha_2 \to 0 \\ \beta_2
\to \infty}}\underset{[0,a]\times[0,a]}{\iint }g_{(x,y)}(t_1,t_2) \\
& \times h_{\alpha_1,\beta_1}(t_1)h_{\alpha_2,\beta_2}(t_2)d(t_1,t_2) \\
=&\lim\limits_{\substack{ \alpha_1 \to 0 \\ \beta_1 \to \infty}}%
\lim_{a\rightarrow\infty}\int^a_0 \frac{g_{(x,y)}(t_1,0+)}{2}%
h_{\alpha_1,\beta_1}(t_1)dt_1 \\
=&\lim\limits_{\substack{ \alpha_1 \to 0 \\ \beta_1 \to \infty}}%
\int^{\infty}_0 \frac{g_{(x,y)}(t_1,0+)}{2}h_{\alpha_1,\beta_1}(t_1)dt_1 \\
=&\frac{g_{(x,y)}(0+,0+)}{4},
\end{split}%
\end{equation*}
where $g_{(x,y)}(0+,0+)=f(x+,y+)+f(x+,y-)+f(x-,y+)+f(x-,y-)$.
\end{proof}

Theorems 1 and 3 in \cite{Moricz2} and Theorem 2.1 in \cite{Ghodadra} claim that Fourier series of functions $f$ in $L^1(\mathbb{R}^2)\cap BV_H(\mathbb{R}^2)$ are bounded and converge locally uniform at points of continuity of $f$. Thus, the extensions of such results for functions in $BV_{||0||}(\mathbb{R}^2)$ are stated as follows. 

\begin{theorem}
\label{TP11} Let $f\in BV_{||0||}(\mathbb{R}^2)$ and $(u_n),(v_m)\in (L)$
with constants $A_1$ and $A_2$, respectively. Then, for each $(x,y)\in%
\mathbb{R}^2$, 
\begin{eqnarray*}
\sum^{\infty}_{i,j=2}\sup_{u\in[u_{i-1},u_i]}\sup_{v\in[v_{j-1}v_j]%
}\left|\frac{1}{4\pi^2}\iint_{R_{u_{i-1},v_{j-1}}^{u,v}}\mathcal{F}%
(f)(\xi,\eta)e^{i(x\xi+y\eta)}d(\xi,\eta)\right|\leq kVar(f,\mathbb{R}^2),
\end{eqnarray*}
where $k=(3A_1+4)(3A_2+4)/ \pi^2$.
\end{theorem}

\begin{proof}
Let $(x,y)\in\mathbb{R}^2$ be. For $u\in[u_{i-1},u_i]$, $v\in[v_{j-1},v_j]$
with $i,j\geq 2$, let 
\begin{eqnarray*}
A_{i,j}(u,v):=\frac{1}{4\pi^2}\iint_{R_{u_{i-1},v_{j-1}}^{u,v}}\mathcal{F}%
(f)(\xi,\eta)e^{i(x\xi+y\eta)}d(\xi,\eta).
\end{eqnarray*}

By Proposition \ref{TY3}, we have that 
\begin{equation*}
\begin{split}
A_{i,j}(u,v)=&\frac{1}{\pi^2}\lim \limits_{\substack{ a,c \to -\infty \\ b,d
\to \infty}} \iint_{[a,b]\times[c,d]}f(x+t_1,y+t_2)\left(%
\int^u_{u_{i-1}}cos(t_1\tau)d\tau\right) \\
&\times\left(\int^v_{v_{j-1}}cos(t_2\tau)d\tau\right) d(t_1,t_2).
\end{split}%
\end{equation*}
Since $f(x+\cdot,y+\cdot)\in BV_{||0||}(\mathbb{R}^2)$ and from
Propositions \ref{TY4} and \ref{TY6}, we obtain that
\begin{eqnarray*}
A_{i,j}(u,v)=\frac{1}{\pi^2}\iint_{\mathbb{R}^2}f(x+t_1,y+t_2)d\left(%
\int^u_{u_{i-1}}\frac{sin(t_1\tau)}{\tau}d\tau\right)\left(\int^v_{v_{j-1}}%
\frac{sin(t_2\tau)}{\tau}d\tau\right) \\
=\frac{1}{\pi^2}\iint_{\mathbb{R}^2}\left(\int^u_{u_{i-1}}\frac{sin(t_1\tau)%
}{\tau}d\tau\right)\left(\int^v_{v_{j-1}}\frac{sin(t_2\tau)}{\tau}%
d\tau\right)df(x+t_1,y+t_2).
\end{eqnarray*}
Now, $ii)$ of Lemma \ref{TG1} assures that 
\begin{equation}  \label{eq0}
\begin{split}
|A_{i,j}(u,v)|\leq& \text{ }\frac{1}{\pi^2}\iint_{\mathbb{R}%
^2}\left|\int^u_{u_{i-1}}\frac{sin(t_1\tau)}{\tau}d\tau\right|\left|%
\int^v_{v_{j-1}}\frac{sin(t_2\tau)}{\tau}d\tau\right|dV_{x,y}(f;t_1,t_2) \\
=:&B(u,v),
\end{split}%
\end{equation}
where $V_{x,y}(f;t_1,t_2)=Var(f,(-\infty,x+t_1]\times(-\infty,x+t_2])$.%
\newline

\noindent Applying Theorem \ref{TY1}, there exists a $\sigma-$algebra $%
\mathbb{M}$ of $\mathbb{R}^2$ and a measure $\mu_{x,y}$ on $\mathbb{M}$
such that 
\begin{equation}
\begin{split}
B(u,v)=& \text{ }\frac{1}{\pi^2}\iint_{\mathbb{R}^2}\left|\int^u_{u_{i-1}}%
\frac{sin(t_1\tau)}{\tau}d\tau\right|\left|\int^v_{v_{j-1}}\frac{sin(t_2\tau)%
}{\tau}d\tau\right|d\mu_{x,y}(t_1,t_2) \\
\leq& \text{ }\frac{1}{\pi^2}\iint_{\mathbb{R}^2}\left(\max_{u\in
[u_{i-1},u_i]}\left|\int^u_{u_{i-1}}\frac{sin(t_1\tau)}{\tau}%
d\tau\right|\right) \\
&\times\left(\max_{v\in [v_{j-1},v_j]}\left|\int^v_{v_{j-1}}\frac{%
sin(t_2\tau)}{\tau}d\tau\right|\right)d\mu_{x,y}(t_1,t_2).
\end{split}
\label{eq1}
\end{equation}
%%%
From (\ref{eq0}) and (\ref{eq1}), for each $u\in[u_{i-1},u_i]$ and $v\in[v_{j-1},v_j]$, the following is satisfied
\begin{equation*}
\begin{split}
|A_{i,j}(u,v)|\leq &\text{ } \frac{1}{\pi^2}\iint_{\mathbb{R}%
^2}\left(\max_{u\in [u_{i-1},u_i]}\left|\int^u_{u_{i-1}}\frac{sin(t_1\tau)}{%
\tau}d\tau\right|\right) \\
&\times\left(\max_{v\in [v_{j-1},v_j]}\left|\int^v_{v_{j-1}}\frac{%
sin(t_2\tau)}{\tau}d\tau\right|\right)d\mu_{x,y}(t_1,t_2).
\end{split}%
\end{equation*}

We define, for each $i,j\geq 2$, 
\begin{eqnarray*}
M_{i,j}(f;x,y):=\sup_{u\in[u_{i-1},u_i]}\sup_{v\in[v_{j-1}v_j]%
}\left|A_{i,j}(u,v)\right|.
\end{eqnarray*}
One can prove, for each $i,j\geq 2$, 
\begin{equation*}
\begin{split}
M_{i,j}(f;x,y)\leq& \text{ } \frac{1}{\pi^2}\iint_{\mathbb{R}%
^2}\left(\max_{u\in [u_{i-1},u_i]}\left|\int^u_{u_{i-1}}\frac{sin(t_1\tau)}{%
\tau}d\tau\right|\right) \\
&\times\left(\max_{v\in [v_{j-1},v_j]}\left|\int^v_{v_{j-1}}\frac{%
sin(t_2\tau)}{\tau}d\tau\right|\right)d\mu_{x,y}(t_1,t_2).  \label{eqM}
\end{split}%
\end{equation*}

Considering the above inequalities and by Lemma \ref%
{LM1}, for $i,j\geq 2$, we obtain that

\begin{equation}  \label{ECP}
\begin{split}
\sum^{\infty}_{i=2}\sum^{\infty}_{j=2}M_{i,j}(f;x,y)\leq& \text{ } \frac{1}{%
\pi^2}\sum^{\infty}_{i=2}\sum^{\infty}_{j=2}\iint_{\mathbb{R}%
^2}\left(\max_{u\in [u_{i-1},u_i]}\left|\int^u_{u_{i-1}}\frac{sin(t_1\tau)}{%
\tau}d\tau\right|\right) \\
&\times\left(\max_{v\in [v_{j-1},v_j]}\left|\int^v_{v_{j-1}}\frac{%
sin(t_2\tau)}{\tau}d\tau\right|\right)d\mu_{x,y}(t_1,t_2) \\
=& \text{ } \frac{1}{\pi^2}\iint_{\mathbb{R}^2}\left(\sum^{\infty}_{i=2}%
\max_{u\in [u_{i-1},u_i]}\left|\int^u_{u_{i-1}}\frac{sin(t_1\tau)}{\tau}%
d\tau\right|\right) \\
&\times\left(\sum^{\infty}_{j=2}\max_{v\in
[v_{j-1},v_j]}\left|\int^v_{v_{j-1}}\frac{sin(t_2\tau)}{\tau}%
d\tau\right|\right)d\mu_{x,y}(t_1,t_2) \\
\leq&\text{ } \frac{1}{\pi^2}\iint_{\mathbb{R}^2}(3A_1+4)(3A_2+4)d%
\mu_{x,y}(t_1,t_2) \\
=& \text{ }\frac{1}{\pi^2}(3A_1+4)(3A_2+4)Var(f,\mathbb{R}^2).
\end{split}%
\end{equation}
\end{proof}

\begin{theorem}
\label{Theo 4.2} Let $f\in BV_{||0||}(\mathbb{R}^2)$ and $(x,y)\in\mathbb{R}%
^2$ a point of continuity of $f$. If $(u_n),(v_m)\in (L)$ with constants $A_1
$ y $A_2$, respectably, then the series 
\begin{eqnarray*}
\sum^{\infty}_{i=2}\sum^{\infty}_{j=2}\sup_{u\in[u_{i-1},u_i]}\sup_{v\in[%
v_{j-1}v_j]}\left|\frac{1}{4\pi^2}\iint_{R_{u_{i-1},v_{j-1}}^{u,v}}\mathcal{F}%
(f)(\xi,\eta)e^{i(\cdot \xi+\cdot \eta)}d(\xi,\eta)\right|
\end{eqnarray*}
converges locally uniform at $(x,y)$.
\end{theorem}

\begin{proof}
Let $n,m\in\mathbb{N}$ be. From (\ref{ECP}), for $(x,y)\in%
\mathbb{R}^2$, the following holds 
\begin{equation*}
\begin{split}
\sum^{\infty}_{i=n+1}\sum^{\infty}_{j=m+1}M_{i,j}(f;x,y)\leq \frac{1}{\pi^2}%
\iint_{\mathbb{R}^2}\left(\sum^{\infty}_{i=n+1}\max_{u\in
[u_{i-1},u_i]}\left|\int^u_{u_{i-1}}\frac{sin(t_1\tau)}{\tau}%
d\tau\right|\right) \\
\times\left(\sum^{\infty}_{j=m+1}\max_{v\in
[v_{j-1},v_j]}\left|\int^v_{v_{j-1}}\frac{sin(t_2\tau)}{\tau}%
d\tau\right|\right)d\mu_{x,y}(t_1,t_2),  \label{eq2}
\end{split}%
\end{equation*}
where $\mu_{x,y}$ is a measure satisfying 
\begin{equation}  \label{ED1}
\mu_{x,y}((a,b)\times (c,d))\leq Var(f,[a,b]\times [c,d]+(x,y)) 
\end{equation}
\begin{equation}  \label{ED2}
\text{ and  } \mu_{x,y}(\mathbb{R}^2)\leq Var(f,\mathbb{R}^2).
\end{equation}

Suppose that $(x_0,y_0)\in\mathbb{R}^2$ is a point of continuity of $f$ and $%
\varepsilon>0$. By Lemma \ref{TG2}, there exits $\delta>\delta_0=\delta/2>0$%
, such that 
\begin{equation*}
Var(f,[x_0-2\delta_0,x+2\delta_0]\times[y_0-2\delta_0,y_0+2\delta_0]%
)<\varepsilon.
\end{equation*}
Now, if $(x^\prime,y^\prime)\in(x_0-\delta_0,x_0+\delta_0)\times(y_0-%
\delta_0,y_0+\delta_0)$, we can claim that 
\begin{equation*}
\begin{split}
[x^\prime-\delta_0,x^\prime+\delta_0]\times[y^\prime-\delta_0,y^\prime+%
\delta_0]\subset[x_0-2\delta_0,x_0+2\delta_0]\times[y_0-2\delta_0,y_0+2%
\delta_0].
\end{split}%
\end{equation*}
Then, 
\begin{equation}  \label{DVt}
Var(f,[x^\prime-\delta_0,x^\prime+\delta_0]\times[y^\prime-\delta_0,y^%
\prime+\delta_0])<\varepsilon.
\end{equation}

Let $0<\delta_1<\delta_0$ and $(x^\prime,y^\prime)\in(x_0-\delta_0,x_0+%
\delta_0)\times(y_0-\delta_0,y_0+\delta_0)$. From (\ref{eq2}), it follows that 
\begin{equation}
\begin{split}
\sum^{\infty}_{i=n+1}\sum^{\infty}_{j=m+1}M_{i,j}(f;x^\prime,y^\prime)\leq&%
\text{ } \frac{1}{\pi^2}\left(\text{ }\iint\limits_{\substack{ |t_1|\leq
\delta_1 \\ |t_2|\leq \delta_1}}+\iint\limits_{\substack{ |t_1|\geq \delta_1
\\ |t_2|\leq \delta_1}}+\iint\limits_{\substack{ |t_1|\leq \delta_1 \\ %
|t_2|\geq \delta_1}}+\iint\limits_{\substack{ |t_1|\geq \delta_1 \\ %
|t_2|\geq \delta_1}}\right) \\
&\left(\sum^{\infty}_{i=n+1}\max_{u\in [u_{i-1},u_i]}\left|\int^u_{u_{i-1}}%
\frac{sin(t_1\tau)}{\tau}d\tau\right|\right) \\
&\times\left(\sum^{\infty}_{j=m+1}\max_{v\in
[v_{j-1},v_j]}\left|\int^v_{v_{j-1}}\frac{sin(t_2\tau)}{\tau}%
d\tau\right|\right)d\mu_{x^\prime,y^\prime}(t_1,t_2) \\
=:& I_1+I_2+I_3+I_4.  \label{eq4}
\end{split}%
\end{equation}
By (\ref{ED1}), (\ref{DVt}) and Lemma \ref{LM1}, 
\begin{equation}
\begin{split}
I_1\leq &\frac{(3A_1+4)(3A_2+4)}{\pi^2}\iint\limits_{\substack{ |t_1|\leq
\delta_1 \\ |t_2|\leq \delta_1}}d\mu_{x^\prime,y^\prime}(t_1,t_2) \\
\leq& \frac{(3A_1+4)(3A_2+4)}{\pi^2}\mu_{x^\prime,y^\prime}([-\delta_1,%
\delta_1]\times[-\delta_1,\delta_1]) \\
\leq& \frac{(3A_1+4)(3A_2+4)}{\pi^2}\mu_{x^\prime,y^\prime}((-\delta_0,%
\delta_0)\times(-\delta_0,\delta_0)) \\
\leq& \frac{(3A_1+4)(3A_2+4)}{\pi^2}Var(f,[x^\prime-\delta_0,x^\prime+%
\delta_0]\times[y^\prime-\delta_0,y^\prime+\delta_0]) \\
\leq& \frac{(3A_1+4)(3A_2+4)}{\pi^2}\varepsilon.  \label{eq2t5}
\end{split}%
\end{equation}

From Lemmas \ref{LM1}, \ref{LM2} and inequality in (\ref{ED2}),  
\begin{equation}
\begin{split}
I_2\leq &\frac{(3A_2+4)}{\pi^2}\iint\limits_{\substack{ |t_1|\geq \delta_1
\\ |t_2|\leq \delta_1}}\frac{3A_1}{|t_1|u_n}d\mu_{x^\prime,y^\prime}(t_1,t_2)
\\
\leq& \frac{3A_1(3A_2+4)}{\delta_1u_n\pi^2}Var(f,\mathbb{R}^2).
\label{eq2t6}
\end{split}%
\end{equation}

Using similar arguments, we conclude that 
\begin{equation}
\begin{split}
I_3\leq &\frac{(3A_1+4)}{\pi^2}\iint\limits_{\substack{ |t_1|\leq \delta_1
\\ |t_2|\geq \delta_1}}\frac{3A_2}{|t_2|v_m}d\mu_{x^\prime,y^\prime}(t_1,t_2)
\\
\leq& \frac{3A_2(3A_1+4)}{\delta_1v_m\pi^2}Var(f,\mathbb{R}^2).
\label{eq2t7}
\end{split}%
\end{equation}

Applying Lemma \ref{LM2} it follows that 
\begin{equation}
\begin{split}
I_4\leq &\frac{1}{\pi^2}\iint\limits_{\substack{ |t_1|\geq \delta_1 \\ %
|t_2|\geq \delta_1}}\left(\frac{3A_1}{|t_1|u_n}\right)\left(\frac{3A_2}{%
|t_2|v_m}\right)d\mu_{x^\prime,y^\prime}(t_1,t_2) \\
\leq& \frac{9A_1A_2}{\delta^2_1u_nv_m\pi^2}Var(f,\mathbb{R}^2).
\label{eq2t8}
\end{split}%
\end{equation}

Adding the expressions in (\ref{eq2t5}), (\ref{eq2t6}), (\ref{eq2t7}) and (%
\ref{eq2t8}), we have that
\begin{equation*}
\begin{split}
\sum^{\infty}_{i=n+1}\sum^{\infty}_{j=m+1}M_{i,j}(f;x^\prime,y^\prime)\leq& 
\frac{(3A_1+4)(3A_2+4)}{\pi^2} \\
&\times\left(\varepsilon+\left(\frac{1}{\delta_1u_n}+\frac{1}{\delta_1u_m}+%
\frac{1}{\delta^2_1u_nv_m}\right)Var(f,\mathbb{R}^2)\right),
\end{split}%
\end{equation*}
for each $(x^\prime,y^\prime)\in(x_0-\delta_0,x_0+\delta_0)\times(y_0-%
\delta_0,y_0+\delta_0)$ and $n,m\in\mathbb{N}$.\newline

Considering (\ref{eq4}) and $u_n,v_m\rightarrow\infty$, as $n,m\rightarrow\infty$,
respectively, and given $\varepsilon>0$, there exist $N,M\in\mathbb{N}$ such
that $n\geq N$ and $m\geq M$, then 
\begin{equation*}
\begin{split}
\sum^{\infty}_{i=n+1}\sum^{\infty}_{j=m+1}M_{i,j}(f;x^\prime,y^\prime)\leq& 
\frac{(3A_1+4)(3A_2+4)}{\pi^2}4\varepsilon,
\end{split}%
\end{equation*}
for each $(x^\prime,y^\prime)\in(x_0-\delta_0,x_0+\delta_0)\times(y_0-%
\delta_0,y_0+\delta_0)$. This prove the theorem.
\end{proof}

\section{Concluding remark}
It remains to do an exhaustive analysis of the truncated double Fourier integral of functions $f\in BV_{||0||}(\mathbb{R}^2)$ over regions contemplating the coordinate axes. We conjecture that the limit (\ref{TIP}) is locally uniform at every point of continuity of $f\in BV_{||0||}(\mathbb{R}^2)$. 

\section*{Acknowledgment}
 The authors express their sincere gratitude to the referees for their valuable comments and suggestions which help improve this paper.
 
\section*{Funding}
 This research was supported partially by CONAHCyT-SNI, Mexico.


\begin{thebibliography}{99}

\bibitem{Giselle} G. Antunes, A. Slav\'ik, M. Tvrd\'y,  Henstock-Stieltjes Integral Theory and Applications, World Scientists, Singapure, 2017.

\bibitem{Bary} N. K. Bary, A Treatise on Trigonometric Series,
Pergamon, New York, 1964.

\bibitem{Clarkson} J. A. Clarkson, On double Riemann-Stieltjes integral, Bull. Amer. Math. Soc. 39 (1933) 929-936. https://doi.org/10.1090/S0002-9904-1933-05771-3.

\bibitem{Ghodadra} B. L. Ghodadra, V. F$\ddot{\text{u}}$l$\ddot{\text{o}}$p, On the convergence of double Fourier integrals of functions of bounded variation on $\mathbb{R}^2$, Stud. Sci. Math. Hung. 53 (2016) 289-313. https://doi.org/10.1556/012.2016.53.3.1336.

\bibitem{Kadelburg} Z. Kadelburg, M. Marjanovi\'c, Interchanging two limits, Teach. Math. 8 (2005) 15-29.

\bibitem{Lee} T. Y. Lee, Henstock-Kurzweil Integration on Euclidean Spaces, World Scientific Publishing, Singapore, 2011.

\bibitem{Muldowney} P. Muldowney, V. Skvortsov, Improper Riemann Integral and Henstock Integral in $\mathbb{R}^n$, Math. Notes 78 (2002) 228-233. https://doi.org/10.1007/s11006-005-0119-7.

\bibitem{Mendoza2} F. J. Mendoza, E. Torres, U. Sengul, A contribution to the Dirichlet-Jordan theorem for non Lebesgue integrable functions, Eurasian Bull. Math. 3 (2020) 114-126. 

\bibitem{Mendoza3} F. J. Mendoza, J. H. Arredondo, S. S\'anchez-Perales, O. Flores-Medina, E. Torres-Teutle, The double Fourier transform of non-Lebesgue integrable functions of bounded Hardy-Krause variation, Georgian Math. J. 30 (2023) 403-415. https://doi.org/10.1515/gmj-2023-2008.

%%%%%

\bibitem{Moricz1} F. M\'oricz, Pointwise behavior of Fourier integrals of functions of bounded variation over $\mathbb{R}$, J. Math. Anal. Appl. 297 (2004) 527-539. https://doi.org/10.1016/j.jmaa.2004.03.025.

\bibitem{Moricz2} F. M\'oricz, Pointwise convergence of double
Fourier integrals of functions of bounded variation over $\mathbb{R}^2$,
J. Math. Anal. Appl. 424 (2015) 1530-1543. https://doi.org/10.1016/j.jmaa.2014.12.007.

\bibitem{Pringsheim} A. Pringsheim, Elementare Theorie der unendliche Doppel-reihen, Sitzungsber. Akad. Wiss. 27 (1897) 101-153.

\bibitem{Rudin} W. Rudin, Real and Complex Analysis, McGraw-Hill, Singapure, 1987.

\end{thebibliography}
\end{document}